\documentclass{article}

\usepackage[T1]{fontenc}
\usepackage[english]{babel}

\usepackage[a4paper,top=2cm,bottom=2cm,left=2cm,right=2cm,marginparwidth=1.75cm]{geometry}

\usepackage{amsmath}
\usepackage{comment}
\usepackage{bbold}
\usepackage{dsfont}
\usepackage{graphicx}
\usepackage[colorlinks=true, allcolors=blue]{hyperref}
\usepackage{amsthm}
\usepackage{wrapfig}
\usepackage{amssymb}
\usepackage{stmaryrd}
\usepackage{csquotes}
\usepackage{biblatex}
\usepackage{orcidlink}
\newtheorem{prop}{Proposition}
\newtheorem{theorem}[prop]{Theorem}
\newtheorem{lemma}[prop]{Lemma}

\newtheorem*{remark}{Remark}

\newtheorem{corollary}[prop]{Corollary}

\newcommand{\Mod}{\,\mathrm{mod}\,}
\newcommand{\e}{\mathrm{e}}
\newcommand{\de}{\,\mathrm{d}}
\newcommand{\ii}{\mathrm{i}}

\newcommand{\A}{\mathcal{A}}
\newcommand{\B}{\mathcal{B}}

\newcommand{\Z}{\mathds{Z}}

\newcommand{\R}{\mathds{R}}

\newcommand{\Ind}{\mathbb{1}}
\newcommand{\eps}{\varepsilon}

\title{Exponential sums over primes are unbounded}
\author{P\'ea (Pierre-Alexandre) Bazin \orcidlink{0009-0000-7184-9130} \footnote{Universit\'e Paris Cit\'e, Sorbonne Universit\'e, CNRS \\ Institut de Math\'ematiques de Jussieu-Paris Rive Gauche, 75013 Paris, France. \\ Email: bazin@imj-prg.fr}}
\date{}

\begin{document}
\maketitle

\begin{abstract}
    We prove prime exponential sums have no better than square root cancellation on average on short intervals, in the sense that $$\frac{1}{x} \sum_{-y< n\le x} \left|\sum_{\substack{n< m \le n+y\\ 1\le m \le x}} \Lambda(m) \e(\alpha m)\right|^2 \gg y\log y$$ whenever $y \ll x^{1/3-\eps}.$ This answers a question of Aymone and Ramar\'e by proving the lower bound $$\sup_{n\le x} \left|\sum_{m\le n} \Lambda(m) \e(\alpha m)\right| \gg x^{1/6 - \eps}.$$ 
\end{abstract}

\section{Introduction}
\subsection{Main results}

The main purpose of this paper is to answer the following question of Aymone, which was communicated to the author by Ramar\'e at the July 2024 Italy-France Analytic Number Theory Workshop in Genoa: is the prime exponential sum $$\pi(x; \alpha) := \sum_{p\le x} \e(\alpha p)$$ (where we use the standard notation $\e(\beta) := \e^{2\ii\pi\beta}$ and the sum is on primes $p$) unbounded for all $\alpha$ ?

We answer positively to the question and obtain a quantitative lower bound (Corollary \ref{main-cor} below). To get this bound, we develop a new method (see Theorem \ref{initial-prop} below) to obtain lower bounds on exponential sums of the form $\sum_n f(n) \e(\alpha n).$ In this paper we apply it to $f= \Lambda$ (Theorem \ref{main-uniform}) as well as the divisor function $f = \tau,$ where our method gives nearly optimal results (Theorem \ref{uniform-tau}). In a paper with Ihor Pylaiev and Fred Tyrrell \cite{tyrrell}, we apply it to get qualitative results for general multiplicative $f,$ whereas in \cite{pi_k}, we develop results to apply it for $f(n) = \Ind_{\Omega(n) = k}.$ We first present our results for $f= \Lambda$ and $f = \tau$ before stating our general results.

\subsection{The case of primes}

We study here the exponential sum $\psi(x;\alpha) := \sum_{n\le x} \Lambda(n) \e(\alpha n),$ which is closely related to $\pi(x; \alpha).$
The study of this exponential sum dates back to Vinogradov \cite[§9]{vinogradov}, who proved the first non-trivial \textit{upper bound} on $\psi(x;\alpha).$  
\begin{theorem}[Vinogradov \cite{vinogradov}, §9, Theorem 3]\label{vino}
    Define for $\alpha \in \R/\Z$ and $x\ge 2,$ \begin{equation}\label{def-R}
        R(x,\alpha) := \min\left\{R > 0 : \exists\, 1\le q\le R, (a,q) = 1, \left|\alpha - \frac{a}{q}\right| \le \frac{R}{qx}\right\}.
    \end{equation} Then when $\eps > 0$ is fixed, we have uniformly in $\alpha \in \R/\Z$ and $x \ge 2,$ $$\psi(x;\alpha) \ll \frac{x^{1+\eps}}{R(x,\alpha)^{1/2}} + x^{4/5+\eps}.$$
\end{theorem}
In a forthcoming paper (announced at the 2026 Probability in Number Theory conference in Montreal), Maynard improved the $x^{4/5+\eps}$ term to $x^{19/24+\eps}.$
\begin{remark}
    By Dirichlet's theorem, we always have $R(x,\alpha) \le x^{1/2}$ and, for algebraic irrational $\alpha,$ we have $R(x,\alpha) \gg x^{1/2-\eps}.$
    When $R(x,\alpha)$ is very small, we say $\alpha$ is in a major arc and the study of $\psi(x;\alpha)$ reduces to the study of primes in arithmetic progressions mod $q$ (where $q$ comes from (\ref{def-R})). When $R(x,\alpha)$ is close to its typical value $\asymp x^{1/2}$, Vinogradov's theorem gives an upper bound $\psi(x; \alpha) \ll x^{4/5+\eps}$ (and Maynard's improvement gives $\psi(x;\alpha) \ll x^{19/24 + \eps}$). 
\end{remark}

However, much less is known regarding \textit{lower bounds} for $\psi(x; \alpha)$ when $R(x;\alpha)$ is large. The most significant lower bound result is that of Vaughan \cite{vaughan}, who showed the $L^1$ average lower bound $$\int_0^1 \big|\psi(x; \alpha)\big| \de\alpha \gg x^{1/2}$$ (note that we trivially have the $L^2$ average $\int_0^1 \big|\psi(x; \alpha)\big|^2 \de\alpha \asymp x\log x$ developing the square and using the prime number theorem). However, no non-trivial lower bound is known for fixed $\alpha.$ 


Our main result proves the prime exponential sum has no better than square root cancellation on short intervals.

\begin{theorem}\label{main-uniform}
    Let $\eps >0.$ Then for $y \le x^{1/3 - \eps},$ we have $$\inf_{\alpha\in\R} \frac{1}{x} \sum_{-y < n\le x} \left|\sum_{\substack{n< m\le n+y \\ 1 \le m \le x}} \Lambda(m) \e(\alpha m)\right|^2 \gg y\log x,$$ where the implied constant depends only on $\eps.$
\end{theorem}

The following corollary then answers Aymone and Ramar\'e's question.

\begin{corollary}\label{main-cor}
    For all $\eps > 0,$ we have $$\inf_{\alpha\in\R} \sup_{n\le x} \big|\psi(n; \alpha) \big| \gg x^{1/6 - \eps}$$ and $$\inf_{\alpha\in\R} \sup_{n\le x} \big|\pi(n; \alpha) \big| \gg x^{1/6 - \eps}.$$
\end{corollary}

It is well known Diophantine approximations of $\alpha$ play a key role in the behaviour of $\psi(x; \alpha),$ as can be seen by the role of $R(x,\alpha)$ in Theorem \ref{vino}. Similarly, when we have a Diophantine approximation of suitable height (which happens when $R(y,\alpha) \asymp y^{1/2}$) and assuming the Generalized Riemann Hypothesis (GRH) for all Dirichlet $L$-functions, we can improve on Theorem \ref{main-uniform}.

\begin{theorem}\label{main-limsup}
    Assume GRH. Fix $\alpha\in\R$ and $\eps, \eps' > 0.$ Then for $y \le x^{5/9 - \eps}$ such that we have a rational approximation $\left|\alpha - \frac{u}{s}\right| \le \frac{1}{s^2}$ for some irreducible fraction $\frac{u}{s}$ with $\eps' s^2 \le y \le s^2/12,$ we have $$\frac{1}{x} \sum_{-y < n\le x} \left|\sum_{\substack{n< m\le n+y \\ 1 \le m \le x}} \Lambda(m) \e(\alpha m)\right|^2 \gg y\log x,$$ where the implied constant depends only on $\eps$ and $\eps'.$
\end{theorem}
\begin{corollary}\label{grh-cor}
    Assume GRH. Then for all irrational $\alpha,$ we have $$\limsup_{n \to \infty} n^{- 5/18 +\eps} \big|\psi(n;\alpha)\big| > 0.$$
\end{corollary}

We note that the assumption on $\alpha$ in Theorem \ref{main-limsup} plays a key role in going beyond the $x^{1/3}$ barrier of Theorem~\ref{main-uniform}, as discussed in Section \ref{y_ge_q}. However, when this condition is verified, we do not need an assumption as strong as GRH to get beyond $y = x^{1/3},$ even though the best known character sum estimates barely fail to improve the range of $y.$

\subsection{The case of divisors}

For the sake of illustration, we now consider the case of the divisor function $\tau,$ where our method gives results that are within a $x^{o(1)}$ factor of the truth.

\begin{theorem}\label{uniform-tau}
    Let $\alpha\in\R$ and $\eps > 0.$ Then there is an absolute constant $C$ such that for $y \le Y := x\left(\frac{\log\log x}{C\log x}\right)^2,$ whenever we have a rational approximation $\left|\alpha - \frac{u}{s}\right| \le \frac{1}{s^2}$ for some irreducible fraction $\frac{u}{s}$ with $\eps s^2 \le y \le s^2/12,$ we have $$\frac{1}{x} \sum_{-y < n\le x} \left|\sum_{\substack{n< m\le n+y \\ 1 \le m \le x}} \tau(m) \e(\alpha m)\right|^2 \gg y \big(\log (x/y)\big)^2 \log(Y/y),$$ where the implied constant depends only on $\eps.$ In particular, for all irrational $\alpha,$ we have
    $$\limsup_{x \to \infty} x^{-1/2}(\log x)(\log\log x)^{-2} \left|\sum_{n\le x} \tau(n) \e(\alpha n)\right| > 0.$$
\end{theorem}

For comparison, Goldston and Pandey \cite{goldston} have shown the $L^{1}$ lower bound $$\int_0^1 \left|\sum_{n\le x} \tau(n) \e(\alpha n)\right| \de\alpha \gg x^{1/2},$$ but no non-trivial lower bound for fixed irrational $\alpha$ is known (note however than the exponential sum is trivially unbounded as $\tau$ is unbounded).

\begin{remark}
    As pointed out by Gabdullin (private communication), it follows from the Carleson--Hunt inequality \cite{hunt} that Theorem \ref{uniform-tau} is sharp up to a logarithmic factor for almost all $\alpha.$ Indeed, the Carleson--Hunt inequality applied for $\tau$ shows $$\int_0^1 \sup_{N\le x} \left|\sum_{n\le N} \tau(n) \e(\alpha n)\right|^2 \de \alpha \ll \sum_{n\le x} \tau(n)^2 \asymp x(\log x)^3,$$ so that for $\eps>0$ we have $$\sup_{N\le x} \left|\sum_{n\le N} \tau(n) \e(\alpha n)\right|^2 \le x(\log x)^{4+\eps}$$ for all $\alpha \in [0,1]$ outside of a set of measure $\ll (\log x)^{1+\eps}.$ It then follows considering $x = 2^k$ and using the Borel-Cantelli lemma that $$\limsup_{x \to \infty} x^{-1/2}(\log x)^{-2-\eps} \left|\sum_{n\le x} \tau(n) \e(\alpha n)\right| = 0$$ for almost all $\alpha.$ This shows Theorem \ref{uniform-tau} can be improved by a logarithmic factor at best. A similar upper bound can be obtained when $\tau$ is replaced by $\Lambda$ by the same argument.
\end{remark}

\subsection{Notation}

We fix $\alpha\in\R$ and $f$ an arithmetic function, which will later be taken to be $\Lambda$ or $\tau.$ We define the exponential sum $$F(x;\alpha) := \sum_{n\le x} f(n)\e(\alpha n),$$ as well as $F(x) := F(x; 0) = \sum_{n\le x} f(n).$ For $f=1,$ we note $$E_x(\alpha) := \sum_{1\le n\le x} \e(\alpha n).$$

We then define the sums on short intervals, $$F(n, n+y; \alpha) := F(n+y;\alpha) - F(n;\alpha) = \sum_{n< m\le n+y} \e(\alpha m)f(m).$$ When $f = \Lambda,$ we write $\psi$ instead of $F.$

Finally, when $Q,y < x$ are fixed, we define 
\begin{equation}\label{def-A}
    \A = \A(Q,y) := \left\{\frac{a}{q} : q \le Q, (a, q) =1, \left|\alpha - \frac{a}{q}\right| \le \frac{1}{6y} \right\}
\end{equation}
the set of rational approximations of $\alpha$ of height $Q$ with precision $1/y,$ which plays a key role in our estimates.

Throughout the paper, the implicit constants from the $\ll$ and $O()$ notations may depend on constants $\eps$ and $C$ when relevant, but not on $\alpha.$

\subsection{Outline of the proofs}\label{sec-main-prop}

Our core idea is to look at the behaviour of the exponential sum in a random short interval. Indeed, fixing $f$ supported on $[1,x],$ we expect $F(n,n+y; \alpha)$ to behave like $F(n, n+y; a/q)$ whenever $\left|\alpha - \frac{a}{q}\right| \ll 1/y,$ and are able to control the error terms on average as soon as $q$ is small in terms of $x.$ Thus, instead of requiring $R(x,\alpha)$ to be small, we only need $R(y,\alpha)$ to be small (in terms of $x$), which can be ensured by taking $y$ small enough.

Our tool to estimate the average $\sum_n \big|F(n,n+y ; \alpha)\big|^2$ is the large sieve (reminded below), which when applied to $n \mapsto F(n, n+y;\alpha)$ gives us a lower bound on this average based on the value of its Fourier transform at chosen points.

\begin{theorem}[Montgomery-Vaughan large sieve, \cite{mv}]\label{large-sieve}
    Let $M, N\in\Z, N\ge 0,$ $f$ supported on $[M+1, M+N],$ and $\B \subset \R/\Z$ a finite set, with $$\delta := \min_{\beta \ne \beta'\in \B} |\beta - \beta'|.$$ Then 
    \begin{equation}
        \sum_{\beta\in\B} \big|F(M, M+N; \beta)\big|^2 \le \big(N+\delta^{-1}\big)\sum_{n = M+1}^{M+N} \big|f(n)\big|^2.
    \end{equation}
\end{theorem}

We can now state and prove our main inequality from which we derive Theorems \ref{main-uniform}, \ref{main-limsup} and \ref{uniform-tau}. 

\begin{theorem}\label{initial-prop}
    Let $f$ supported on $[1,x].$ Then for $y \le x$ and $Q \le x^{1/2},$ we have 
    \begin{equation}\label{initial}
        \sum_{-y < n \le x} \big|F(n, n+y; \alpha)\big|^2 \gg \frac{1}{x} \sum_{\substack{q\le Q \\ (a,q) = 1}} \big|F(x; a/q)\big|^2 \left|E_y\left(\alpha - \frac{a}{q}\right)\right|^2 \gg \frac{y^2}{x} \sum_{\frac{a}{q}\in\A} \big|F(x; a/q)\big|^2,
    \end{equation}
    where $\A = \A(Q,y)$ is defined by (\ref{def-A}).
\end{theorem}

\begin{proof}
    We fix $g : n\mapsto F(n, n+y;\alpha),$ which is supported on $[-y+1, x-1]$ and whose Fourier transform is given for $\beta\in\R$ by 
    \begin{align*}
        \sum_{-y< n < x} g(n)\e(\beta n) &= \sum_{-y < n < x} \left(\sum_{n < m \le n+y} f(m) \e(m\alpha)\right) \e(n\beta)
        \\&= \sum_{1\le m\le x} \sum_{1\le h \le y} f(m) \e(m\alpha) \e\big((m-h)\beta\big)  
        \\&= F(x; \alpha + \beta) E_y(-\beta),
    \end{align*}
    since $f$ is supported on $[1,x].$
    Applying Theorem \ref{large-sieve} to $g$ and $$\B := \left\{\frac{a}{q} - \alpha : q\le Q, (a,q) = 1\right\}$$ (which verifies $\delta \ge Q^{-2}$) thus gives 

    \begin{align*}
        \sum_{-y < n \le x} \big|F(n, n+y; \alpha)\big|^2 &\ge \frac{1}{\lfloor x\rfloor+\lfloor y\rfloor+Q^2}\sum_{q\le Q} \sum_{(a,q) = 1} \big|F(x; a/q)\big|^2 \left|E_y\left(\alpha - \frac{a}{q}\right)\right|^2 
        \\&\gg \frac{y^2}{x} \sum_{\frac{a}{q}\in\A} \big|F(x; a/q)\big|^2 
    \end{align*}
    as $y \le x, Q \le x^{1/2}$ and $E_y(\beta) \gg y\Ind_{|\beta|\le 1/6y}.$
\end{proof}

We now need to decide the relation between $y$ and $Q.$ When $y = Q^{1-\eps},$ we can find major arcs $\frac{a}{q}\in\A$ for every denominator $6y \le q \le Q.$ In this case we are able to estimate the right-hand side of (\ref{initial}) using a Bombieri-Vinogradov type estimate for exponential sums. In the case of primes (Theorem \ref{main-uniform}), we use the following estimate by Liu and Zhan (note that we actually only need the case $\lambda = 0$).

\begin{theorem}[{Liu and Zhan \cite[Theorem 3]{liuzhan}}]\label{liuzhan}
    For all $A > 0,$ there is a constant $E = E(A) > 0$ such that for all $x \ge 3$ and $$Q \le x^{1/3}(\log x)^{-E},\, \delta \le Q^{-3}(\log x)^{-E},$$ we have 
    $$\sum_{q\le Q} \max_{(a,q) = 1} \max_{|\lambda| < \delta} \left|\psi\left(x; a/q + \lambda\right) - \frac{\mu(q)}{\phi(q)} E_x(\lambda)\right| \ll x(\log x)^{-A}.$$
\end{theorem}

Since Theorem \ref{liuzhan} requires $Q \ll x^{1/3},$ and only provides an accurate estimate of the right-hand side of (\ref{initial}) when $y \ll Q,$ we need the bound $y \ll x^{1/3}$ in Theorem \ref{main-uniform}.

However, when $Q \ll y \ll Q^2,$ we still expect to find some fractions $\frac{a}{q}$ with $q\le Q$ and $\left|\alpha - \frac{a}{q}\right| \ll \frac{1}{6y}.$ This is why when we assume information on \textit{all} $F(x; a/q)$ (given by GRH in Theorem \ref{main-limsup}), we are able to obtain results for larger $y$ such as Theorem \ref{main-limsup}.

\section{Major arcs in short intervals}

\subsection{The setting $y = Q^{1-\eps}$}

When we consider major arcs of denominator up to $Q = y^{1+\eps},$ we get the following result, from which we derive Theorem \ref{main-uniform} in Section \ref{bv}.

\begin{prop}\label{y-le-q}
    Assume $f$ is supported on $[1,x],$ fix $Q \le x^{1/2}$ and assume there is a constant $C$ and functions $g$ and $\eta$ (which may depend on $x$) with $\eta(q) \ll (\log q)^C$ such that 
    \begin{equation}\label{bv-e}
        \sum_{q\le Q} |g(q)| \max_{(a,q) = 1} \left|F(x; a/q) - x\frac{g(q)}{\phi(q)} \right| \ll \frac{x}{(\log x)^{C}}
    \end{equation} and for all $\eps > 0,$
    \begin{equation}\label{goal-g}
        \sum_{Q^{1-\eps}< q \le Q} \frac{|g(q)|^2}{\phi(q)} \gg \frac{\log Q}{\eta(Q)}.
    \end{equation}
    Then for $\eps > 0$ and $y \le Q^{1-\eps},$ we have $$\frac{1}{x}\sum_{-y< n\le x} \big|F(n, n+y;\alpha)\big|^2 \gg \frac{y\log Q}{\eta(Q)}.$$
\end{prop}

\begin{remark}
    When $f = \Lambda,$ we have (\ref{bv-e}) with $g=\mu$ from Theorem \ref{liuzhan}, so that (\ref{goal-g}) holds for $\eta = 1.$ When $f$ is multiplicative, we expect $g \approx f * \mu$ and (\ref{goal-g}) might only hold true for a larger $\eta.$ An extreme case is $f = 1,$ for which $g = \Ind_{q=1}$ and (\ref{goal-g}) does not hold. This is consistent with the fact the exponential sum $\sum_{n\le x} \e(\alpha n)$ is bounded in this case. 
\end{remark}

\begin{proof}[Proof of Proposition \ref{y-le-q}]
    For technical reasons, it will be useful to only consider major arcs with denominators $\gg y,$ so we fix $Q_0 := Q^{1 - \eps/2} \ge y^{1+\eps/2}$ and define \begin{equation}\label{def-A1}
        \A_1 := \left\{\frac{a}{q} : Q_0< q \le Q, (a,q) =1, \left|\alpha - \frac{a}{q}\right| \le \frac{1}{6y} \right\} \subseteq \A.
    \end{equation} 
    
    Since $Q \le x^{1/2},$ we have by Theorem \ref{initial-prop} 
    \begin{align}
        \frac{1}{x}\sum_{-y< n\le x} \big|F(n, n+y;\alpha)\big|^2 &\gg
        \frac{y^2}{x^2} \sum_{\frac aq\in\A_1} x^2\frac{|g(q)|^2}{\phi(q)^2} - 2x\frac{|g(q)|}{\phi(q)} \left|F(x; a/q) - x\frac{g(q)}{\phi(q)}\right| \notag
        \\&= y^2 \sum_{\frac aq\in\A_1} \frac{|g(q)|^2}{\phi(q)^2} - E, \label{step8}
    \end{align}
    with an error term $$E := \frac{2y^2}{x} \sum_{\frac aq\in\A_1} \frac{|g(q)|}{\phi(q)} \left|F(x; a/q) - x\frac{g(q)}{\phi(q)}\right|.$$

    Now, since $Q_0 \ge y^{1+\eps/2},$ we have for any $Q_0< q \le Q,$  \begin{equation}\label{coprime}
        \#\left\{a : (a,q) = 1, \left|\alpha - \frac{a}{q}\right| \le \frac{1}{6y}\right\} = \frac{\phi(q)}{3y} + O\big(2^{\omega(q)}\big) \asymp \frac{\phi(q)}{y},
    \end{equation}
    so that the error term is bounded by
    \begin{equation}
        E \ll \frac{y^2}{x} \sum_{Q_0< q\le Q}\frac{\phi(q)}{y} \frac{|g(q)|}{\phi(q)} \max_{(a,q) = 1} \left|F(x; a/q) - x\frac{g(q)}{\phi(q)}\right| \ll \frac{y}{(\log Q)^C} \label{error8}
    \end{equation} using (\ref{bv-e}).

    Moreover, by (\ref{coprime}) and (\ref{goal-g}), the main term of (\ref{step8}) verifies
    \begin{equation}
        \sum_{\frac{a}{q}\in\A_1} \frac{|g(q)|^2}{\phi(q)^2} = \sum_{Q_0 < q \le Q} \frac{|g(q)|^2}{\phi(q)^2} \frac{\phi(q)}{y} \gg \frac{1}{y\eta(Q)},
    \end{equation}
    so that it follows from (\ref{step8}) and (\ref{error8}) $$\frac{1}{x}\sum_{-y< n\le x} \big|F(n, n+y;\alpha)\big|^2 \gg \frac{y\log Q}{\eta(Q)} + O\left(\frac{y}{(\log Q)^C}\right)$$
    as desired. The condition $\eta(Q) \ll (\log Q)^C$ ensures the error term is small enough compared to the main term.
\end{proof}

\subsection{The setting $y \gg Q$} \label{y_ge_q}

We now assume we are able to estimate $F(x;a/q)$ for all $q$ (rather than on average over $q$) and want to take $y \gg Q,$ in order to prove Theorem \ref{main-limsup} which allows for $y \gg x^{1/3}.$ The first problem is that, in that case, there might actually be no major arc in the interval $\left[\alpha - \frac{1}{6y}, \alpha + \frac{1}{6y}\right]$. Indeed, consider $\alpha = \frac{u}{s} + \frac{1}{2y},$ with $s < \frac{y}{Q}.$ Then, any fraction with denominator less than $Q$ (other than $\frac{u}{s}$ itself) is at least $\frac{1}{Qs} > \frac{1}{y}$ away from $\frac{u}{s},$ and thus at least $\frac{1}{2y}$ away from $\alpha,$ which means the set $\A$ of major arcs defined in (\ref{def-A}) is empty. This is the reason why we are not able to go beyond Theorem \ref{main-uniform} without extra information on $\alpha.$

However, for most choices of $\alpha$ and $y,$ the set $\A$ is of the expected cardinality $\asymp Q^2/y$ when $Q \gg y^{1/2+\eps}.$ Indeed, it follows from the lemma below that the set $\A$ is of the expected size for $y$ up to $O(Q^2)$ under the assumption of Theorem \ref{main-limsup}, which occurs for some arbitrarily large values of $y$ for all $\alpha.$ 

\begin{lemma}\label{ratapprox}
    Assume $\left|\alpha - \frac{u}{s}\right| \le \frac{1}{s^2}$ for some irreducible fraction $\frac{u}{s}.$ Then there is an absolute constant $C$ such that for $Q \ge Cs,$ we have $$\#\left\{\frac{a}{q} : Q < q \le 2Q, \left|\alpha - \frac{a}{q}\right| \le \frac{2}{s^2}\right\} \asymp \frac{Q^2}{s^2}.$$
\end{lemma}
\begin{proof}
    The upper bound is trivial from the fact distinct fractions $\frac{a}{q}, \frac{a'}{q'}$ with $Q < q,q' \le 2Q$ are always at least $\frac{1}{qq'} \ge \frac{1}{4Q^2}$ apart.
    Now, since $(u,s) = 1,$ we can consider a B\'ezout relation $vs - tu = 1$ with $s \le t < 2s$ (this is possible as $t$ is defined modulo $s$).  We then have $\left|\frac{u}{s}-\frac{v}{t}\right| = \frac{1}{st} \le \frac{1}{s^2},$ so that both $\frac{u}{s}$ and $\frac{v}{t}$ are less than $\frac{2}{s^2}$ away from $\alpha.$ 

    We now consider \begin{equation*}
        \B := \left\{(a;b) : \frac{Q}{s+t} < a, b \le \frac{2Q}{s+t}, (a,b) = 1 \right\}.
    \end{equation*}

    First, note that $(au+bv,as + bt)$ divides both $(as+bt)v - (au+bv)t = a$ and $(au+bv)s - (as+bt)u = b,$ so that $(au+bv,as + bt) = 1$ when $(a,b) = 1.$
    
    Moreover, $\frac{au+bv}{as+bt}$ is always between $\frac{u}{s}$ and $\frac{v}{t},$ so that $$\left\{\frac{au+bv}{as+bt} : (a;b) \in \B\right\} \subset \left\{\frac{a}{q} : Q < q \le 2Q, \left|\alpha - \frac{a}{q}\right| \le \frac{2}{s^2}\right\}.$$ The conclusion then follows from noting $|\B| \asymp \left(\frac{Q}{s+t}\right)^2 \asymp Q^2/s^2$ when $Q/s \ge C.$
\end{proof}

Showing $\A$ is of the right size is however not sufficient to obtain the main term of Theorem \ref{main-limsup}. Indeed, when $f = \Lambda,$ only the major arcs with \textit{squarefree} denominators yield a significant contribution to the right-hand side of (\ref{initial}) since the expected contribution of a major arc $\frac{a}{q}$ is proportional to $\frac{\mu^2(q)}{\phi^2(q)}.$ We thus use in the proof of Theorem \ref{main-limsup} a result of Heath-Brown, which shows we can find such major arcs with squarefree denominators in the desired interval when $y\ll Q^{5/3-\eps}.$ This explains the bound $y \ll x^{5/9 - \eps}$ appearing in Theorem \ref{main-limsup}, since we need $Q \ll x^{1/3 - \eps}.$

\begin{lemma}[Heath-Brown, \cite{hb}]\label{squarefree-den-lemma}
    Fix $\eps >0.$ Then for all irrational $\alpha,$ there exists infinitely many rationals $\frac{a}{q}$ such that $\mu^2(q) = 1$ and $\left|\alpha - \frac{a}{q}\right| \le q^{-5/3+\eps}.$
    More precisely, when we consider a Diophantine approximation $\left|\alpha - \frac{u}{s}\right| \le \frac{1}{s^2}$ for some irreducible fraction $\frac{u}{s},$ we have for any $Q \gg s^{6/5+\eps},$

    $$\#\left\{\frac{a}{q} : Q < q \le 2Q, \mu^2(q) = 1, \left|\alpha - \frac{a}{q}\right| \le \frac{2}{s^2}\right\} \asymp \frac{Q^2}{s^2}.$$
\end{lemma}

Finally, it turns out information on average on the error term such as (\ref{bv-e}) is too weak to deduce Theorem~\ref{main-limsup}, since $\A$ is too sparse when $y \gg Q.$ This is why we instead assume GRH in Theorem~\ref{main-limsup}, which allows us to estimate $\psi(x;a/q)$ for all major arcs $a/q.$

\section{Proofs of the main results}\label{bv}

In this section, we derive Theorem \ref{main-uniform} from Proposition \ref{y-le-q}, as well as Theorems \ref{main-limsup} and \ref{uniform-tau} using Theorem~\ref{initial-prop}. We then derive Corollaries \ref{main-cor} and \ref{grh-cor} from Theorems \ref{main-uniform} and \ref{main-limsup}.

\subsection{The case of primes}

We first prove Theorem \ref{main-uniform} using Proposition \ref{y-le-q} and Theorem \ref{liuzhan}.

\begin{proof}[Proof of Theorem \ref{main-uniform}]
    Since Proposition \ref{y-le-q} requires a function supported on $[1,x],$ we consider $f = \Lambda\cdot\Ind_{[1,x]}.$ Theorem \ref{liuzhan} then tells us the bound (\ref{bv-e}) is verified for $Q = x^{1/3 - \eps/2} \ge y^{1+\eps/2}$ with $g:=\mu$ (since in that case $|g(q)| \le 1$), and the bound (\ref{goal-g}) is then verified with $\eta = 1.$ Then, by Proposition \ref{y-le-q}, we obtain $$\frac{1}{x} \sum_{-y < n\le x} \left|\sum_{\substack{n< m\le n+y \\ 1 \le m \le x}} \Lambda(m) \e(\alpha m)\right|^2 \gg y\log Q$$ as desired (the extra condition $1 \le m \le x$ comes from the fact we are considering $f = \Lambda\cdot\Ind_{[1,x]}$ rather than~$\Lambda.$)
\end{proof}

We now prove Theorem \ref{main-limsup}. We use the following lemma, which gives us information on exponential sums in major arcs assuming GRH.

\begin{lemma}\label{grh}
    Assume GRH. Then for all $q\le x$ and $(a,q) = 1,$ we have $$\left|\psi(x; a/q) - x\frac{\mu(q)}{\phi(q)}\right| \ll (xq)^{1/2}(\log x)^2.$$
\end{lemma}
\begin{proof}
    We decompose the exponential sum into character sums. Defining the Gauss sum $$G(\chi) := \sum_{(b,q) = 1} \e(b/q) \chi(b),$$ we have when $n$ is coprime with $q,$ \begin{align*}
        \e(an/q) &= \frac{1}{\phi(q)} \sum_{(b,q) = 1} \e(b/q) \sum_{\chi \Mod q} \overline{\chi}(b) \chi(an)
        \\&= \frac{1}{\phi(q)} \sum_{\chi\Mod q} \chi(a) G\left(\overline{\chi}\right) \chi(n),
    \end{align*}
    so that summing on $n$ and noting the contribution of primes that are not coprime with $q$ to $\psi(x; a/q)$ is $\ll \omega(q)\cdot \log x \ll (\log x)^2,$ we deduce
    \begin{equation}\label{decomp}
        \psi(x; a/q) = \frac{1}{\phi(q)} \sum_{\chi\Mod q} \chi(a) G\left(\overline{\chi}\right) \psi(x; \chi) + O\big((\log x)^2\big)
    \end{equation}

    Now, we have $G(\chi) \ll q^{1/2}$ for all $\chi$ (\cite[chapter 9]{davenport}) and, under GRH, we know (\cite[Theorem 13.7]{mv-book}) $$\psi(x,\chi) = x\Ind_{\chi = \chi_0} + O\left(x^{1/2} (\log x)^{2}\right).$$ Injecting these estimates in equation (\ref{decomp}), we get $$\psi(x; a/q) = x\frac{G(\chi_0)}{\phi(q)} + O\big(q^{1/2}x^{1/2}(\log x)^2\big)$$ as desired, as $$G(\chi_0) = \sum_{(b,q) = 1} \e(b/q) = \mu(q).$$
\end{proof}

We can now prove Theorem \ref{main-limsup} using Theorem~\ref{initial-prop} and Lemma~\ref{squarefree-den-lemma}.

\begin{proof}[Proof of Theorem \ref{main-limsup}]
    We again consider $f = \Lambda \cdot \Ind_{[1,x]},$ and fix $Q := x^{1/3 - \eps/6}$ in order to define $\A$ as in (\ref{def-A}). By Theorem \ref{initial-prop}, we then have \begin{align}
        \frac{1}{x} \sum_{-y < n\le x} \left|\sum_{\substack{n< m\le n+y \\ 1 \le m \le x}} \Lambda(m) \e(\alpha m)\right|^2 
        &\gg 
        \frac{y^2}{x^2} \sum_{\frac{a}{q}\in\A} x^2\frac{\mu^2(q)}{\phi^2(q)} - 2x \frac{\mu^2(q)}{\phi(q)} \left|\psi(x; a/q) - x\frac{\mu(q)}{\phi(q)}\right| \notag
        \\&\gg y^2 \sum_{\frac{a}{q} \in\A} \frac{\mu^2(q)}{\phi^2(q)} + O\left(\frac{y^2}{x}\sum_{\frac{a}{q}\in\A} \frac{(xq)^{1/2}(\log x)^{2}}{\phi(q)}\right) \label{step17}
    \end{align}
    where the error term is bounded by Lemma \ref{grh}.

    Now, since we assumed $\eps' s^2 \le y \le s^2/12$ for some rational approximation $\left|\alpha - \frac{u}{s}\right| \le \frac{1}{s^2},$ we have by Lemma~\ref{squarefree-den-lemma} for all $y^{3/5 + \eps/6} \le Q' \le Q/2,$ $$\sum_{\substack{\frac{a}{q}\in\A \\ Q'< q \le 2Q'}} \frac{\mu^2(q)}{\phi^2(q)} \ge \frac{1}{4Q'^2} \#\left\{\frac{a}{q} : Q' < q \le 2Q', \mu^2(q) = 1, \left|\alpha - \frac{a}{q}\right| \le \frac{2}{s^2}\right\} \gg \frac{1}{s^2} \asymp \frac{1}{y},$$
    so that summing on dyadic intervals (as $y^{3/5 + \eps/6} \le Qx^{-\eps/6}$ when $y \le x^{5/9-\eps}$) we obtain the main term of~(\ref{step17}) satisfies 

    $$y^2 \sum_{\frac{a}{q} \in\A} \frac{\mu^2(q)}{\phi^2(q)} \gg y\log x.$$

    Moreover, since the elements of $\A$ are rationals of denominator $\le Q,$ the intervals $\left[\frac{a}{q} - \frac{1}{2qQ}, \frac{a}{q} + \frac{1}{2qQ}\right]$ do not overlap so that $$\sum_{\frac{a}{q}\in\A} \frac{1}{qQ} \le \frac{1}{3y}.$$ Injecting this into the error term of (\ref{step17}), we find that the error term is $\ll y x^{-1/2+\eps/6} Q^{3/2}$ which is smaller than the main term with our choice of $Q := x^{1/3 - \eps/6}.$
\end{proof}

\subsection{The case of the number of divisors}

We now prove Theorem \ref{uniform-tau}. The reason why we are able to reach $y = x^{1-\eps}$ in this setting is because the exponential sum at rationals can be very well understood, as shown by the lemma below.

\begin{lemma}\label{tau-rational}
    For $1\le q \le x^{1/2}$ and $a$ with $(a,q) = 1,$ we have 
    $$\sum_{n\le x} \tau(n) \e\left(\frac{a}{q}n\right) = \frac{x}{q}\Big(\log(x/q^2) + 2\gamma - 1\Big) + O\left(x^{1/2}(1+\log q)\right).$$
\end{lemma}
\begin{proof}
    Writing $\tau(n) = \sum_{uv = n} 1,$ we have 
    \begin{align}
        \sum_{n\le x} \tau(n) \e\left(\frac{a}{q}n\right) &=
        \sum_{mn\le x} \e\left(\frac{a}{q}mn\right) \notag \\&=
        2 \sum_{m\le x^{1/2}} \sum_{m < n \le x/m} \e\left(\frac{am}{q}n\right) + \sum_{m\le x^{1/2}} \e\left(\frac{a}{q}m^2\right)
        \label{step15} \\&:= 2T + E. \notag
    \end{align}

    The sum over squares $E$ is trivially bounded by $O(x^{1/2})$ (better bounds exist, but this one is sufficient for our purpose). Now, if $am \equiv b \Mod q$ with $-q/2 < b \le q/2,$ the inner sum $\sum_{m < n \le x/m} \e\left(\frac{am}{q}n\right)$ is $O(q/|b|)$ unless $b = 0$ in which case it is $x/m - m + O(1).$ Thus, since $a$ is coprime with $q,$ noting $\Bar{a}$ its inverse modulo $q$ we have

    \begin{align*}
        T &= \sum_{\substack{m\le x^{1/2} \\ q|m}} \left(\frac{x}{m}-m+O(1)\right) + \sum_{\substack{-q/2 < b \le q/2 \\ b\ne 0}} \sum_{\substack{m\le x^{1/2} \\ m \equiv \Bar{a}b \Mod q}} O\left(\frac{q}{|b|}\right)
        \\&= \frac{x}{q}\Big(\log (x^{1/2}/q) + \gamma + O(q/x^{1/2})\Big) - \frac{q}{2}\left(\frac{x^{1/2}}{q} + O(1)\right)^2 + O(x^{1/2}/q) + O\left(q \log q \cdot \frac{x^{1/2}}{q}\right)
        \\&= \frac{x}{q}\Big(\log(x^{1/2}/q) + \gamma - 1/2\Big) + O\left(x^{1/2}(1+\log q)\right)
    \end{align*}
    when $1\le q\le x^{1/2}.$ Combining with (\ref{step15}) gives the desired result.
\end{proof}

We can now prove Theorem \ref{uniform-tau} using Theorem \ref{initial-prop} and Lemma \ref{ratapprox}.

\begin{proof}[Proof of Theorem \ref{uniform-tau}]
    We fix $Q := x^{1/2}\frac{\log\log x}{C_1\log x}$ with $C_1$ large enough so that Lemma \ref{tau-rational} yields $$\sum_{n\le x} \tau(n) \e\left(\frac{a}{q}n\right) \asymp \frac{x}{q}\log(x/q^2)$$ for $q \le Q.$ We then have by Theorem \ref{initial-prop} applied to $f = \tau\cdot \Ind_{[1,x]}$ and Lemma \ref{tau-rational}, 
    \begin{equation}\label{step14}
        \frac{1}{x} \sum_{-y < n\le x} \left|\sum_{\substack{n< m\le n+y \\ 1 \le m \le x}} \tau(m) \e(\alpha m)\right|^2 \gg \frac{y^2}{x^2} \sum_{\frac{a}{q}\in\A} \left(\frac{x \log(x/q^2)}{q}\right)^2,
    \end{equation}
    where $\A$ is defined by (\ref{def-A}).

    We now cut the sum into dyadic intervals based on the value of $q.$ Since we have a rational approximation $\left|\alpha - \frac{u}{s}\right| \le \frac{1}{s^2}$ with $\frac{2}{s^2} \le \frac{1}{6y}$ and $s^2 \asymp y,$ we have by Lemma \ref{ratapprox} an absolute constant $C_2$ such that for $C_2 y^{1/2} \le Q' \le Q/2,$ we have 
    \begin{align*}
        \sum_{\substack{\frac{a}{q}\in\A \\ Q' < q \le 2Q'}} \left(\frac{x \log(x/q^2)}{q}\right)^2 &\gg \left(\frac{x \log(x/Q'^2)}{Q'}\right)^2 \#\left\{\frac{a}{q} : Q' < q \le 2Q', (a,q) = 1, \left|\alpha - \frac{a}{q}\right| \le \frac{1}{6y}\right\} \\&\gg \frac{x^2}{y} \big(\log(x/Q'^2)\big)^2.
    \end{align*}
    The conclusion then follows from (\ref{step14}) when $y \le Y := (Q/4C_2)^2$ as the number of dyadic intervals verifying $C_2 y^{1/2} \le Q' \le Q/2$ is then $\gg \log(Y/y).$
\end{proof}

\subsection{Proofs of Corollaries}

We now prove Corollary \ref{main-cor} from Theorem \ref{main-uniform} and Corollary \ref{grh-cor} from Theorem \ref{main-limsup}.

\begin{proof}[Proof of Corollary \ref{main-cor}]
    Fix $y = x^{1/3 - \eps}$ and let $\alpha\in\R.$ We have for all $-y < n \le x,$
    \begin{equation}\label{L2_le_sup}
        \left|\sum_{\substack{n< m\le n+y \\ 1 \le m \le x}} \Lambda(m) \e(\alpha m)\right| = \big|\psi\big(\min\{n+y,x\}; \alpha\big) - \psi(n; \alpha)\big| \le 2\sup_{n\le x} \big|\psi(n; \alpha) \big|.
    \end{equation}
    Taking the $L^2$ norm, we get \begin{equation}\label{step2}
        \sup_{n\le x} \big|\psi(n; \alpha) \big| \gg \left(\frac{1}{x} \sum_{-y < n\le x} \left|\sum_{\substack{n< m\le n+y \\ 1 \le m \le x}} \Lambda(m) \e(\alpha m)\right|^2\right)^{1/2} \gg (y\log x)^{1/2} \gg x^{1/6-\eps}
    \end{equation} uniformly in $\alpha$ by Theorem \ref{main-uniform} as desired. 

    To move on to $\pi(n,\alpha),$ we note that we have for all $n\ge y^2,$
    $$\sum_{\substack{n< m\le n+y \\ 1\le m\le x}} \Lambda(m) \e(\alpha m) = \sum_{\substack{n< p\le n+y \\ 2\le p\le x}} (\log p) \e(\alpha p) + O\big((\log n)^2\big) = (\log n)\sum_{\substack{n< p\le n+y \\ 2\le p\le x}} \e(\alpha p) + O\big((\log n)^2\big)$$ since there is at most one $p^k$ in $(n, n+y]$ for each $k\ge 2,$ and $\log p = \log n + O(y/n) = \log n + O(1/y)$ for $n<p\le n+y.$

    Thus, \begin{align*}
        \sup_{n\le x} \big|\pi(n;\alpha)\big| &\ge \frac{1}{2} \sup_{y^2\le n\le x} \left|\sum_{\substack{n< p\le n+y \\ 1\le p\le x}} \e(\alpha p)\right|
        \\&\ge \frac{1}{2\log x} \sup_{y^2\le n\le x} \left|\sum_{\substack{n< m\le n+y \\ 1\le m\le x}} \Lambda(m) \e(\alpha m)\right| + O\left(\log x\right)
        \\&\ge \frac{1}{2\log x} \left(\frac{1}{x} \sum_{y^2 < n\le x} \left|\sum_{\substack{n< m\le n+y \\ 1 \le m \le x}} \Lambda(m) \e(\alpha m)\right|^2\right)^{1/2} + O\left(\log x\right)
        \\&\gg \frac{1}{\log x}\big(y\log x + O(y^4/x)\big)^{1/2} + O\left(\log x\right)
        \\&\gg x^{1/6-\eps}
    \end{align*}
    by Theorem \ref{main-uniform} with $y = x^{1/3-\eps}.$
\end{proof}

We now prove Corollary \ref{grh-cor}, where we assumed $\alpha$ is irrational and GRH in order to use Theorem \ref{main-limsup}.

\begin{proof}[Proof of Corollary \ref{grh-cor}]
    Since $\alpha$ is irrational, we can find arbitrarily large $s$ such that $\left|\alpha - \frac{u}{s}\right| \le \frac{1}{s^2}$ for some irreducible fraction $\frac{u}{s}.$ For such an $s,$ fix $y = s^2/12$ and $x = y^{9/5+\eps}.$ It then follows from (\ref{L2_le_sup}) and Theorem \ref{main-limsup} that for such $x,$ we have $n \le x$ such that $$\big|\psi(n;\alpha)\big| \gg (y\log x)^{1/2} \gg n^{5/18 - \eps}.$$ 
    Moreover, we must have $$n \ge \big|\psi(n;\alpha)\big| \gg (y\log x)^{1/2} \gg s,$$ which ensures $n$ grows to infinity with $s.$
\end{proof}

\section*{Acknowledgements}

I would like to thank my advisor R\'egis de la Bret\`eche (Universit\'e Paris Cit\'e) for his guidance. 
I would also like to thank Olivier Ramar\'e (Aix-Marseille Universit\'e) for presenting me this problem and Sary Drappeau (Universit\'e Clermont Auvergne) for discussing parts of the problem with me. Finally, I would like to thank Mikhail Gabdullin (University of Illinois) for the remark on the Carleson--Hunt theorem.

\printbibliography

\end{document}